%% file: pisa2303.tex
\documentclass[10pt]{article}

\usepackage{amsmath}   
\usepackage{amssymb}   
\usepackage{stmaryrd}  
\usepackage{latexsym}  
\usepackage{amsxtra}   
\usepackage{amstext}   
\usepackage{bm}        
\usepackage{amsthm}    
\usepackage{amscd}     

\usepackage{epsfig}    %
\usepackage{color}

\usepackage[latin1]{inputenc}

\usepackage{epic}      
\usepackage{eepic}     
\usepackage{calc}      

\usepackage{xy}        
\xyoption{all}

\theoremstyle{plain}
\theoremstyle{definition}

\newtheorem{aussage}{Aussage}[section]

\newtheorem{theorem}[aussage]{Theorem}
\newtheorem{lemma}[aussage]{Lemma}

\numberwithin{equation}{section}

\newcommand{\bF}{\mathbb{F}}

\newcommand{\bN}{\mathbb{N}}

\newcommand{\bQ}{\mathbb{Q}}
\newcommand{\bR}{\mathbb{R}}

\newcommand{\bZ}{\mathbb{Z}}

\newcommand{\cB}{\mathcal{B}}

\newcommand{\cD}{\mathcal{D}}

\newcommand{\cN}{\mathcal{N}}

\newcommand{\Aut}{\text {Aut}}
\newcommand{\Hom}{\text {Hom}}
\newcommand{\lbracket}{\langle}
\newcommand{\rbracket}{\rangle}

\begin{document}


\title{EMBEDDINGS OF BRAID GROUPS INTO MAPPING CLASS GROUPS AND THEIR HOMOLOGY} 


\bigskip
\bigskip

\author{Carl-Friedrich B\"odigheimer and Ulrike Tillmann} 

\bigskip

\date{}


\maketitle

\bigskip

\begin{abstract}
\noindent
We construct several families of
embeddings  of braid groups into mapping 
class groups of orientable and non-orientable surfaces and prove 
that they induce the trivial map in stable homology in the  orientable case, 
but not so in the  non-orientable case. 
We show that these embeddings are non-geometric in the sense that the 
standard generators of the braid group are not mapped to Dehn twists.
\end{abstract}

\bigskip

\bigskip
\section{Introduction}

\indent
Let $\Gamma_{g, n}$ denote the mapping class group of an oriented surface 
$\Sigma_{g,n}$ of genus $g$ with $n$ parametrized boundary components, i.e.,
$\Gamma_{g,n}$ is the group of connected components of the group of orientation
preserving diffeomorphisms of $\Sigma_{g,n}$ that fix the boundary point-wise.
For a simple closed curve $a$ on the surface, let $D_a$ denote
the Dehn twist around $a$. When two simple closed curves $a$ and $b$
intersect in one point the associated Dehn twists satisfy the braid relation
$D_a D_b D_a = D_b D_a D_b$, and if they do not intersect, the corresponding 
Dehn twists commute $D_a D_b = D_b D_a$.


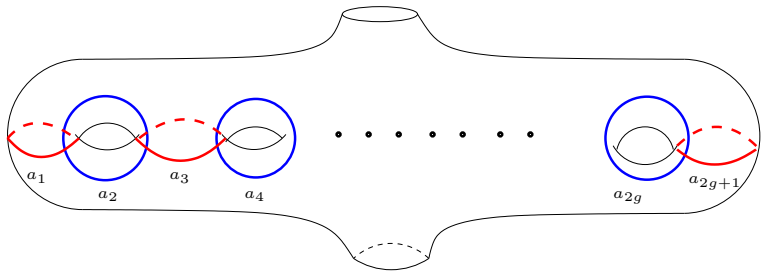
\begin{figure}[!htbp!]
	\centering
	\input{fig1.pstex_t}
	\caption{\em The standard geometric embedding $\phi: B_{2g+2} \to \Gamma _{g,2}$.}
\end{figure}

\bigskip

\noindent
Thus a chain of $n$ interlocking simple closed curves 
$a_1, \dots, a_n$ on some surface 
defines a map from the braid 
group $B_{n+1}$ on $n+1$ strands into the mapping class
group of a subsurface $\Sigma$ containing the union of theses curves; 
these mapping classes
fix the boundary of $\Sigma$ point-wise. 
The smallest such subsurface $\Sigma $ is a 
neighbourhood of the union of the curves. 
When $n=2g+1$ is odd, this is $\Sigma_{g,2}$, and
when $n=2g$ is even, this is $\Sigma _{g,1}$. 
Thus we have homomorphisms of groups

\begin{equation}
\phi: B_{2g+2} \longrightarrow \Gamma _{g,2} \quad \text { and }\quad \\
\phi: B_{2g+1} \longrightarrow \Gamma _{g,1}.
\end{equation}

\noindent
These are injections by a theorem of Birman and Hilden [BH1], [BH2].

\medskip

Wajnryb [W1] calls such embeddings 
that send the standard generators of the braid group 
to Dehn twists {\it geometric}. He 
asks in [W2]  whether there are non-geometric embeddings.
The first example of such a non-geometric embedding was given by 
Szepietowski [S]. Our first goal in this paper 
is to produce many more such non-geometric embeddings
and show that they are ubiquitous.  
We construct these in section 2 and prove in section 3 that they are 
non-geometric.

\medskip

Motivated by a conjecture of Harer and 
following some ideas of F.R. Cohen [C], in [SoT] and [SeT]
it was shown that the geometric embedding $\phi$ induces the trivial map
in stable homology, that is the map in homology is zero in positive degrees
as long as the genus of the underlying surface is large enough relative 
to the degree.  
Our second goal here is to show that this is also the case 
for the non-geometric embeddings constructed in section 2. 
(For one of these maps, this answers a question left open in [SoT].) 
To this purpose
we show in section 4 how our embeddings from 
section 2 induce maps of algebras over an $E_2$-operad, and deduce
in section 5
that all of them  induce the trivial map in stable homology.
Furthermore, while it may be expected that all maps from a braid group 
to the mapping class group of an orientable surface will induce the trivial 
map on stable homology, we show that
this is not true for embeddings of a
braid group into the mapping class group of a non-orientable surface 
by explicitly computing the image of one such embedding in stable homology.  
Finally, in section 6,  we analyse the induced maps in unstable homology.
Only partial results are obtained here. In particular, 
it remains an open question whether $\phi$
induces the trivial map in unstable homology for
field coefficients.

\bigskip
{\bf Acknowledgement.}
We would like to thank Blazej Szepietowski for sending his paper and 
Mustafa Korkmaz for e-mail correspondence.

\bigskip
\section{Non-geometric embeddings}

We will construct various injections of braid groups into mapping class groups. 
All embeddings that we know of, geometric or not, 
initially start with the standard 
identification of the braid groups
with mapping class groups. The pure braid group on $g$ strands 
is the mapping class group
$\Gamma _{0,1}^g$ of a disk with $g$ marked points, 
and the pure ribbon braid group is the mapping
class group $\Gamma _{0,g+ 1} \simeq \bZ^g \times \Gamma _{0,1}^g$ of a disk 
with $g$ holes the boundary of which are 
parametrised. The factors of $\bZ$ correspond to the Dehn twists around the 
boundary circles of the holes. Similarly, the braid group $B_g$ can be 
identified with the mapping class group $\Gamma _{0, 1} ^ {(g)}$
of a disk  with
$g$ punctures (or $g$ unordered marked points),
 and the ribbon braid group $\bZ \wr B_g$ with the mapping class group 
$\Gamma_{0, (g), 1}$ of the disk with $g$ unordered holes for which
the underlying diffeomorphisms may in a parametrisation 
preserving way interchange the  boundaries of the holes and fix 
the outer boundary curve point-wise.
Note that the resulting inclusion 

\begin{equation}
\gamma: B_g \hookrightarrow \bZ \wr B_g \simeq \Gamma _{0, (g), 1}
\end{equation}

\noindent
maps the standard generator that interchanges the $i$-th and $(i+1)$-st 
strand in the braid group to half of the Dehn twists around a simple closed curve 
enclosing the $i$-th and $(i+1)$-st holes followed by a half Dehn 
twist around each of these holes in the opposite direction; see Figure 2.


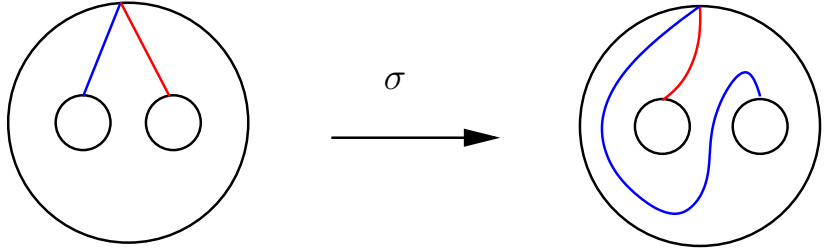
\begin{figure}[!htbp!]
	\centering
	\input{fig2.pstex_t}
	\caption{\em Image of a generator $\sigma \in B_g$ under $\gamma$.}
\end{figure}


\bigskip

Clearly, any genus zero subsurface $\Sigma _{0, g+1}$
of a surface $\Sigma$  defines for us a map 
from the pure ribbon braid group $\Gamma_{0,g+1}$ into the mapping 
class group of $\Sigma$. Not all such maps, however, can be extended to the 
ribbon braid group $\Gamma _{0, (g),1}$. Below we explore various 
constructions of surfaces from $\Sigma_{0,g+1}$ that allow such an extension.

\bigskip
\subsection{Mirror construction.} 

Our first example of a non-geometric embedding
was also considered in [SoT]. 

\medskip

We double the disk with $g$
holes by reflecting it in a plane containing the boundary circles of the
holes to obtain an oriented  surface $\Sigma_{g-1, 2}$ 
as indicated in Figure 3, and  extend diffeomorphisms of the disk
with holes to $\Sigma_{g-1,2}$ by reflection in the plane. This defines
a map on mapping class groups
$$
R: \Gamma_{0, (g),1} \longrightarrow \Gamma_{g-1,2}.
$$

\noindent
Precomposing $R$ with
the natural inclusion $B_g \subset
\Gamma _{0, (g),1}$ defines
the map

\begin{equation}
B_g \overset \gamma \longrightarrow \Gamma_{0, (g),1} \overset R \longrightarrow \Gamma_{g-1,2}.
\end{equation}



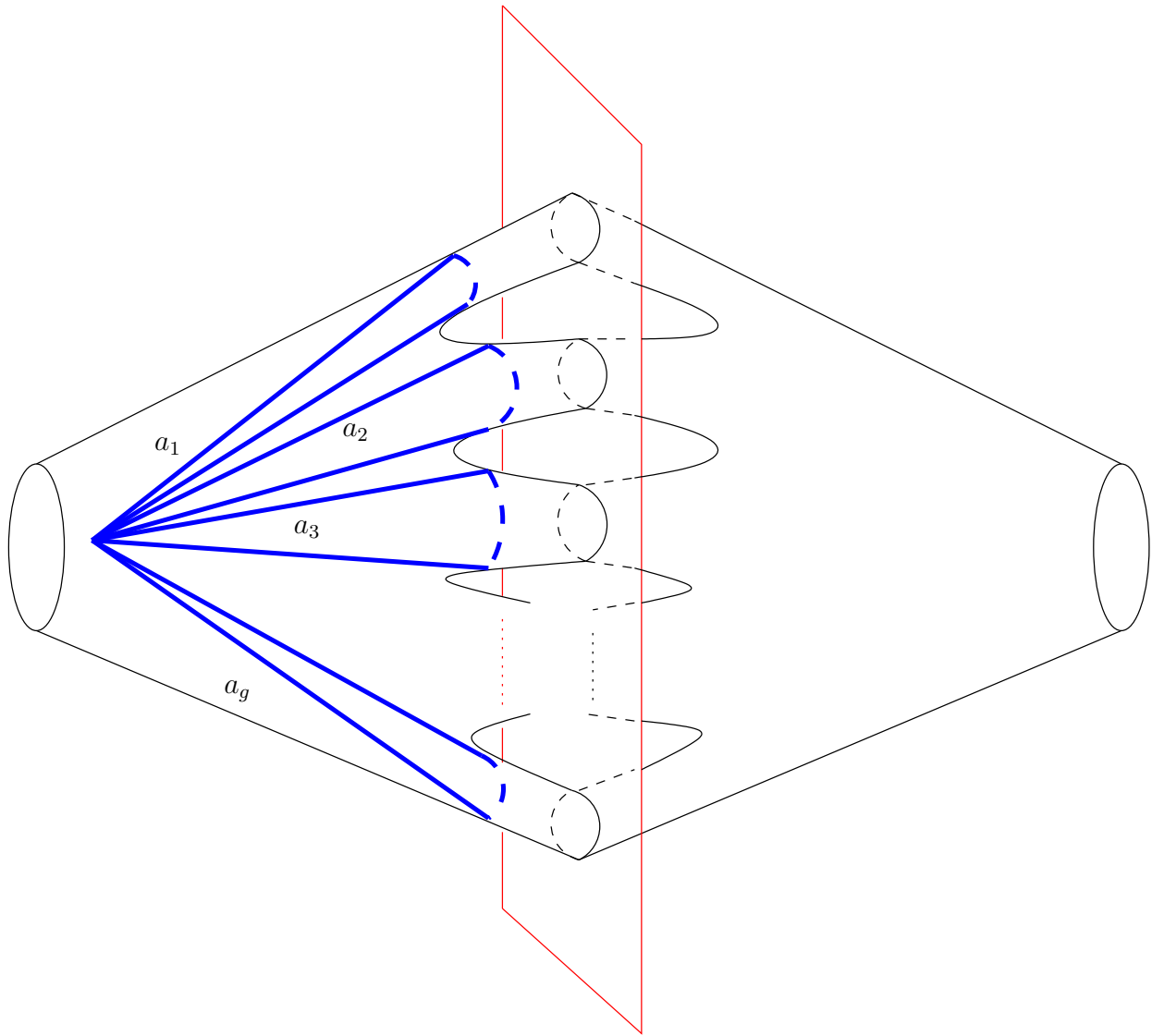
\begin{figure}[!htbp!]
	\centering
	\input{fig3.pstex_t}
	\caption{\em Mirror construction $R$ and subgroup 
$\lbracket a_1, \dots , a_g \rbracket \subset \pi_1(\Sigma _{g-1,2})$.}
\end{figure}


\begin{lemma}
{\it The composition $R \circ \gamma$ is an injection.}
\end{lemma}

\begin{proof} 
We are going to detect the elements in the image by their action 
on the fundamental group of the underlying surface.
Recall, the action
of the braid group on the fundamental group of the disk with $g$ holes
defines Artin's inclusion [A] of the braid group into the automorphism
group of 
$\pi_1 (\Sigma_{0,g+1} ) = F_g = \, \lbracket a_1, \dots, a_g \rbracket$, 
a free group of rank $g$:

$$
A: B_g \hookrightarrow \Aut(F_g).
$$

\noindent
The fundamental group of 
$\Sigma_{0,g+1}$
in turn injects into the fundamental group of $\Sigma _{g-1,2}$. 
Indeed, each standard
generator of $F_g = \, \lbracket a_1, \dots , a_g \rbracket$ 
is mapped to  the corresponding generator of
$\pi_1 (\Sigma _{g-1,2})$ which is the  free group 
$F_{2g-1} = \, \lbracket a_1, b_1, \dots, b_{g-1}, a_{g}\rbracket$
of rank $2g-1$; see Figure 3. 

\medskip

The action of diffeomorphisms on the fundamental group of the surface
induces a group homomorphisms 
$\Gamma_{g-1,2} \to \Aut(F_{2g-1})$,
 and by 
restricting to the subgroup $F_g = \, \lbracket a_1, \dots, a_g \rbracket$ a
map to the set $\Hom(F_g, F_{2g-1})$. We follow this map by the projection map 
to the monoid $\Hom(F_g, F_g)$, 
induced from the group homomorphisms $F_{2g-1} \to F_g$ that maps each 
$a_i$ to itself and each $b_i$ to the identity element.
We have the following commutative diagram; 
note that the right vertical arrow is only a map of sets.

$$
\CD
B_g       @> R \circ \gamma >> \Gamma _{g-1,2}   	\\
@V A VV		                @VVV		        \\
\Aut(F_g)  @>>>	                \Hom (F_g, F_g).
\endCD
$$

\medskip

\noindent
As $A$ and the bottom horizontal map are injective this 
shows that $R \circ \gamma$ is injective.
\end{proof}

\bigskip
\subsection{Szepietowski's construction.}  

We next recall the construction from [S].
Starting with $\Sigma_{0,g+1}$, the disk with $g$ holes,
we  glue to  the boundary of each hole
a M\"obius band $N_{1,1}$ so that 
the resulting surface 
is a non-orientable surface $N_{g,1}$ of genus $g$ with one
parametrised boundary component. 
Any diffeomorphisms $f : \Sigma_{0,g+1} \to \Sigma_{0,g+1}$ which fixes the boundary pointwise
can be extended across the boundary by the identity 
of $N_{1,1} \sqcup \ldots \sqcup N_{1,1}$ and gives thus a diffeomorphism of $N_{g,1}$. 
This defines (via $\gamma$) a homomorphism from the pure braid group to the mapping class group of $N_{g,1}$.
This latter homomorphism can be extended to a map from the whole braid group as follows:
extend a diffeomorphism $f$ which permutes the $g$ inner boundaries by the corresponding permutation diffeomorphism
of $N_{1,1} \sqcup \ldots \sqcup N_{1,1}$.  
This defines 

\begin{equation}
\varphi: B_g \longrightarrow \cN_{g,1},
\end{equation}

\noindent
where $\cN_{g,1}$ denotes the mapping class group of $N_{g,1}$.
By a result of Birman and Chillingworth [BC1], [BC2], extended to surfaces with 
boundary in [S], the lift of diffeomorphisms of a non-orientable surface to its
double cover induces an injection of mapping class groups: 

$$
L: \cN_{g,1} \hookrightarrow \Gamma_{g-1, 2}.
$$

\begin{lemma}
{\it The composition $L \circ \varphi$ is an injection.}
\end{lemma}

\begin{proof} We give an alternative  proof to the one in [S].
As in the proof of Lemma 2.1 we consider the action of the
mapping class group  on 
the fundamental group of the underlying surfaces.
In this case the fundamental group of the disk with $g$ holes
$\pi_1 (\Sigma _{0, g+1}) = F_g = \, \lbracket a_1, \dots, a_g\rbracket$  injects
into the fundamental group $\pi_1(N_{g,1}) =
F_g = \, \lbracket c_1, \dots , c_g \rbracket$
by sending $a_i$ to $2c_i$. 
As by [A] the action of the braid group
on the subgroup $\lbracket a_1, \dots , a_g\rbracket$ is faithful  
it will be so
on the whole group. Hence $\varphi$ is injective, and so is $L \circ \varphi$.
\end{proof}

\bigskip
\subsection{Geometric embedding and orientation cover.} 

The closed non-orientable surface $N_{2g+1}$ can be obtained
by sewing a M\"obius band to the orientable surface $\Sigma _{g,1}$.

\begin{lemma}
{\it The inclusion of surfaces $\Sigma _{g,1} \subset N_{2g+1}$  
induces an inclusion of mapping class groups 
$$
\Gamma_{g,1} / \lbracket D_\partial \rbracket \, \,  \simeq \,\, \Gamma_g^1 \hookrightarrow \cN_{2g+1}.
$$
}
\end{lemma}

\noindent
Here $D_\partial$ denotes the Dehn twist
around the boundary curve in $\Sigma _{g,1}$. This is a central element
in $\Gamma_{g,1}$.
We identify the quotient group
with the mapping class group of a surface with a marked point
(or a surface with one parametrized boundary component).

\bigskip

\noindent
{\bf Remark.} This subgroup of $\cN_{2g+1}$ can be identified with the
point-wise stabiliser of the core of the M\"obius band that has been sewn  
on to $\Sigma _{g,1}$ to form $N_{2g+1}$. For $g>1$ this subgroup
has infinite index but has the same 
virtual cohomological dimension as  the whole group [I],   

$$
\mathrm{vcd} (\Gamma^1_g) = 4g-3 = \mathrm{vcd} (\cN_{2g+1}).
$$ 

\bigskip

\begin{proof}
It is an elementary fact  that the Dehn twist around the boundary of a 
M\"obius strip is isotopic to the identity. 
This implies that 
$D_\partial$ is in the kernel of $\Gamma _{g,1} \to \cN_{2g+1}$.
To show that $D_\partial $ generates the kernel,
consider the composition  

$$
\Gamma _{g,1} \longrightarrow \cN_{2g+1} \overset L \hookrightarrow \Gamma _{2g}.
$$

\noindent
Let $J$ denote the fix-point free orientation reversing involution
of $\Sigma_{2g}$ with quotient $N_{2g+1}$. As in [BC], we embed $\Sigma _{2g}$
in $\bR^3$ symmetrically around the origin and take $J = -\mathrm{Id}$ to be the
reflection through the origin, see Figure 4.
The image of any
element $x \in \Gamma _{g,1}$ is the product of $\bar x$ with $J \bar x J$,
where $\bar x$ acts on the left side of the surface via $x$ and via the
identity on the right side.
So if $x$ is in the kernel then $\bar x^{-1} = J \bar x J$ in $\Gamma _{2g}$.
But $\bar x^{-1}$ can be represented by a diffeomorphism with support
entirely in the left half of the surface and $J \bar x J$ by a diffeomorphism 
with support entirely in the right half of the surface. 
As the diffeomorphisms are isotopic, so must
be their supports. Hence, $\bar x$ has a representing diffeomorphism
supported in a tubular neighbourhood of the boundary, i.e.
$x \in \, \lbracket D_\partial \rbracket$.
\end{proof}



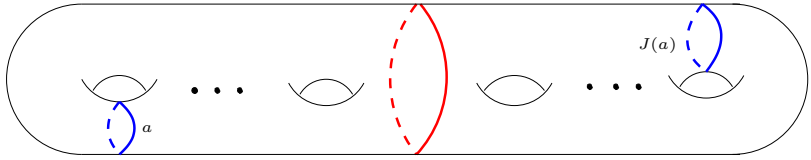
\begin{figure}[!htbp!]
	\centering
	\input{fig4.pstex_t}
	\caption{\em The involution $J = -\mathrm{Id}$ in $\bR^3$.}
\end{figure}

Our third construction is the composition

\begin{equation}
L' \circ \phi: B_{2g}    \overset \phi     \hookrightarrow 
\Gamma_{g-1,2}                             \hookrightarrow
\Gamma_{g}^1                               \longrightarrow  
\cN_{2g+1}                \overset L       \longrightarrow
\Gamma_{2g}.
\end{equation}

\medskip

\noindent
The unlabelled 
map  $\Gamma_{g-1,2} \to \Gamma_{g}^1$, which is induced by gluing a pair
of pants to the two boundary circles of $\Sigma_{g-1, 2}$, is an inclusion, see 
[PR]. Thus, again we have constructed an embedding.

\bigskip
\subsection{Geometric embedding and mirror construction.} 

Similarly to our third example, we may combine the geometric embedding 
$\phi: B_{2g} \to \Gamma _{g-1,2}$ 
with a mirror construction. For this first  glue a torus $\Sigma_{1,2}$ along 
one of its boundary circles to $\Sigma_{g-1,2}$ and embed the other boundary 
circle in the plane. Now double the resulting surface $\Sigma_{g,2}$
by reflection in the plane to yield a surface 
$\Sigma _{2g, 2}$. 
We leave it as an exercise to prove the following result.

\begin{lemma} 
{\it The composition 
$
B_{2g} \overset \phi \longrightarrow \Gamma_{g-1,2} \longrightarrow 
\Gamma_{g, 2} \overset {R'} \longrightarrow \Gamma _{2g,2}
$
is an injection.}
\end{lemma}

\bigskip
\subsection{Operadic embedding.}

The following embedding is well-known as part of an  $E_2$-operad action, 
compare section 4. 

\medskip

Starting with a disk $\Sigma _{0, g+1}$ with $g$ holes we glue
a torus $\Sigma _{1,1}$ with one disk removed to each of boundaries of the $g$
holes of the disk. The result is a surface $\Sigma _{g,1}$. As in the 
construction of $\varphi$ in (2.3) 
we may extend diffeomorphisms of the disk via the identity to the glued on tori
to define a map

\begin{equation}
\varphi^+: B_g \longrightarrow \Gamma_{g,1}.
\end{equation}

\bigskip

\begin{lemma}
{\it The map $\varphi^+$ is an injection.}
\end{lemma}

\begin{proof}
The fundamental group of $\Sigma_{0, g+1}$ 
is freely generated by the $g$ curves $c_1, \dots, c_g$ that start
at a point on the outside boundary and wind around one of the holes: 
$\pi_1 (\Sigma _{0, g+1}) = \, \lbracket c_1, \dots , c_g\rbracket$.
The fundamental group of $\Sigma _{g,1}$ is freely generated by $2g$ curves: 
$\pi_1 (\Sigma _{g,1})= \, \lbracket a_1, b_1, \dots, a_g, b_g \rbracket$ 
with $a_i$ and $b_i$
the standard generators in the fundamental group 
of the $i$-th copy of the torus. 
Then under the inclusion of the disk into the genus $g$ surface
the generator $c_i$ maps 
to $a_i b_i a_i^{-1} b_i^{-1}$. So the $c_i$ are mapped to
words on subsets of the alphabet that are disjoint. 
Hence their images generate
a free group on $g$ generators and the 
induced map on fundamental groups is therefore an injection.
We argue as before, that therefore the  
action of $B_g$ via $\varphi^+$ on $\pi_1(\Sigma _{g,1})$ remains faithful
and hence $\varphi^+$ must be an injection. 
\end{proof}

\bigskip
\subsection{More constructions.}

Other inclusions 
of the braid group into mapping class groups
can be constructed from the above ones
by precomposing with an automorphism of the braid group or
composition with an automorphism of the   mapping class groups.
We note here that conjugation of a non-geometric (or geometric)
embedding by a mapping class yields again a non-geometric (or geometric)
embedding as
Dehn twists are conjugated to Dehn twists.
Thus these will indeed produce new examples 
of non-geometric embeddings (or
geometric ones).
We also note that conjugation by a fixed element 
induces the identity in homology.
The results of section 5 and section 6 are therefore also valid for these variations.
We will not mention these additional embeddings any further.

\bigskip
\section{Proving non-geometricity}

All  of our  constructions in section 2, 
with the exception of $\varphi^+$ in (2.5), use an 
orientation reversing
diffeomorphism of the oriented surface associated to the target.
This is key for proving that these embeddings are not geometric.

\begin{lemma}
{\it Let $J$ be an orientation reversing involution of $\Sigma_{g,n}$ and
$x \in \Gamma _{g,n}$ commute with $J$. Then $x$ is not a power of 
a Dehn twist unless it is trivial.}
\end{lemma}

\begin{proof}
We borrow an argument from [S].
Assume $x$ is the $k$-th  
power of a Dehn twist $D_c$ around a  simple closed curve
$c$ in $\Sigma _{g,n}$. As $x$ 
commutes with $J$ and 
$J$ is orientation reversing, 
$$
x= D_c^k = J D_c^k J = D^{-k}_{J(c)}.
$$
But this identity can only hold if $c$ is isotopic to $J(c)$ and $k=-k$. 
Therefore $x = D_c^k$ is trivial. 
\end{proof}

\begin{theorem} 
{\it The embeddings  $R \circ \gamma$,  $L \circ \varphi$,  $L' \circ \phi$ and
$R' \circ \phi$ are not geometric.}
\end{theorem}

An embedding 
that sends the standard generators of the 
braid group to some powers of Dehn twists are 
also called {\it pseudo-geometric} [W2]. 
The proof of the theorem  will show that these maps 
are not even pseudo-geometric.

\begin{proof} 
Consider $R \circ \gamma$ and 
let $\sigma$ be the image under $\gamma$ of 
one of the standard generators of the 
braid group. 
The image  $R(\sigma)$  is by definition invariant under the reflection in
the plane (see Figure 3) which is orientation reversing.
Lemma 3.1 implies that $R(\sigma)$ cannot be a power of a Dehn twist.
The arguments for $L\circ \varphi$, $L'\circ \phi$ and $R'\circ \phi$
are similar. 
\end{proof}

\bigskip

\noindent
{\bf Remark.} 
The geometric embedding $\phi: B_{2g+2} \to \Gamma_{g,2}$ 
can also be constructed by a \lq doubling'
procedure as the maps  above.
For this,  identify first the braid group $B_{2g+2}$ 
as the mapping class group
of a disk with $2g+2$ unordered marked points which in turn
we identify as the 
orbit space of a genus $g$ surface  
$\Sigma_{g,2}$ under the hyper-elliptic involution, see for example [SeT].
However, in this case the involution is orientation preserving and the 
construction leads to a geometric inclusion.

\bigskip

We now turn to the standard embedding $\gamma$ from (2.1) and
the operadic embedding $\varphi ^+$ constructed in (2.5).

\begin{theorem}
{\it The embeddings $\gamma$ and  $\varphi^+$ are neither  
geometric nor pseudo-geometric.}
\end{theorem}

\begin{proof}
Let $\sigma$ be one of the standard generators for the braid group $B_g$ 
and consider its image under $\gamma$. This is a mapping class supported on 
a disk $\Sigma_{0,3}$ 
with two holes; compare Figure 2. An application of the Jordan 
Curve Theorem shows that there are only three non-contractible
non-isotopic simple closed curves on  
$\Sigma_{0,3}$, each isotopic to one of the boundary circles. It is 
straight forward to check that $\gamma (\sigma)$ is not isotopic to (a power of)
a Dehn twist around any of these three curves. Hence, $\gamma $ is not 
geometric (or pseudo-geometric).

We now turn to $\varphi^+$. By definition (2.5), it is the composition of 
$\gamma$ and the map induced by the inclusion of surfaces
$\Sigma_{0,g+1} \subset \Sigma _{g, 1}$ achieved by sewing a torus with 
one boundary component to each of the interior boundaries of $\Sigma _{0,g+1}$.
Thus $\varphi (\sigma)$ is still defined as pictured in Figure 2 and
supported by the same $\Sigma_{0,3}$, now a 
subsurface of $\Sigma_{g, 1}$. 

The support of any Dehn twist $D_a$ around a simple closed curve 
$a$ is a neighbourhood
of $a$. Thus, if  
$\varphi(\sigma) = D_a^k$ for some $a$ and $k \in \bN$, we must  be able 
to isotope the curve $a$ into $\Sigma _{0,3}$.  But the argument above 
still applies and shows that $\varphi^+(\sigma)$ cannot be (a power of) 
a Dehn twist of any curve in $\Sigma_{0,3}$. Hence $\varphi^+$ is not 
geometric and not pseudo-geometric.
\end{proof}

\bigskip
\section{ Action of the braid group operad}


Consider the following group level version of the well-known $E_2$-operad. 
(See, e.g. [T1] for  details.)

\medskip

As in the introduction to section 2, identify the pure braid group on $k$ strands
with a subgroup  $\cD_k \subset \Gamma _{0,k+1}$ of the pure ribbon braid group, i.e.
the  mapping class group
of a disk with $k$ holes whose boundaries are parametrised. 
The collection $\cD =\{ \cD_k \}_{k\geq 0}$ forms an operad 
with structure maps 

$$
\theta: \cD_k \times (\cD_{m_1} \times \dots \times \cD_{m_k})
\longrightarrow \cD_{m_1 + \dots + m_k}
$$

\noindent
induced by
sewing  of the underlying surfaces. To be more precise, for each $i$,
the  boundary 
of the $i$-th hole in   $\Sigma _{0, k+1}$ 
is sewn to the (outer) boundary of
the $i$-th disk $\Sigma _{0, m_i+1}$.

The operad $\cD$ acts naturally on $\cB = \coprod _{m\geq 1} B_m$
where each braid group $B_m$ is identified via $\gamma$ as a subgroup 
of $\Gamma_{0,(m),1}$. The action is again induced by gluing of the
underlying surfaces. Indeed 

$$
\theta _{\cB}: \cD_k \times (B_{m_1} \times \dots \times B_{m_k}) 
\longrightarrow 
B_{m_1 + \dots + m_k}
$$ 

\noindent
agrees with the structure map $\theta$ on the pure braid subgroups. 

\medskip

This action of $\cD$ can further be  extended to an action on 

$$
\Gamma_R = \coprod_{m>1} \Gamma_{m-1,2}
$$

\noindent
via the mirror construction $R$ from (2.2). To define 

$$
\theta_R: \cD_k \times (\Gamma _{m_1-1, 2} \times \dots \times \Gamma _{m_k-1,2} ) 
\longrightarrow \Gamma _{m_1 + \dots + m_k -1, 2}
$$ 

\noindent
place each of the underlying surfaces across a plane, 
so that
one half is reflected by the plane  onto the other, as in Figure 3. 
Then sewing  $k$-legged trousers to 
the $k$ boundary components on the left halves and 
another one to the right
halves gives a surfaces of type $\Sigma _{m_1 + \dots + m_k -1, 2}$.
An element in $\cD_k$ defines a mapping class on the left $k$-legged
trousers and by mirroring a class on the right
$k$-legged trousers. The following result holds by construction.

\begin{lemma}
{\it The map $R \circ \gamma$ induces a map of $\cD$-algebras $\cB \to \Gamma_R$.}
\end{lemma}

Similarly, $\cN: = \coprod _{m>1} \cN_{m,1}$ is a $\cD$-algebra. Again, the action

$$
\theta _{\cN} : \cD_k \times (\cN_{m_1, 1} \times \dots
\times \cN_{m_k, 1}) \longrightarrow \cN_{m_1 + \dots + m_k, 1}
$$

\noindent
is induced by sewing the legs of  $k$-legged trousers to the boundary
components of the $k$ non-orientable surfaces. And again by construction,
we obtain the following result.

\begin{lemma}
{\it The map $\varphi$ induces a map of $\cD$-algebras $\cB \to \cN$.}
\end{lemma}

To see that the lift $L$ to the orientation cover is a map of $\cD$-algebras 
we need to consider a variant $\Gamma_L$ of  the $\cD$-algebra
$\Gamma_R$. The underlying groups are the same but the action 
$\theta_L$ is such that it commutes with $L$. To achieve this
the mapping class defined on  one $k$-legged trousers is paired with 
that on the other via the lift $L$ so that the following holds.

\begin{lemma}
{\it The map $L$ induces a map of $\cD$-algebras $\cN \to \Gamma_L$.}
\end{lemma}

We recall from [SeT] that also the 
standard embedding $\phi$ from (1.1) 
induces a map of $\cD$-algebras $\cB^{ev} \to \Gamma_\phi$. 
Here $\cB^{ev} = \coprod_{m>1} B_{2m}$ is a $\cD$-subalgebra of $\cB$,
and  $\Gamma_\phi$  is the same collection of  groups as $\Gamma_R$ 
and $\Gamma_L$ but 
has a slightly different action. 
We think of $\phi$ as explained in the remark following Theorem 3.2 above
as  lifting mapping classes of the $2m$-punctured disk 
to the ramified double
cover $\Sigma _{m-1}$ associated to the hyper-elliptic involution. 
Thus, in this case
$\theta_{\phi}$ is defined so as to commute with 
the hyper-elliptic 
involution.

\begin{lemma}
{\it The map $\phi$ induces a map of $\cD$-algebras $\cB^{ev} \to \Gamma_\phi$.}
\end{lemma}

And we also recall the best-known $\cD$-algebra structure on 

$$
\Gamma := \coprod_{m \geq 0} \Gamma_{m,1};
$$ 

\noindent
see [M] and also [B\"o]. 
In this case the action $\theta_\Gamma$
is defined just as for $\theta_{\cN}$ but with 
$\cN_{m_i, 1}$ replaced by $\Gamma_{m_i,1}$. 
As $\varphi^+$ is essentially part of the $\cD$-algebra structure the 
following is immediate.

\begin{lemma}
{\it The map $\varphi^+$ induces a map of $\cD $-algebras $\cB \to \Gamma$.}
\end{lemma}

\bigskip
\section{The maps in stable homology}

We will  determine the induced maps in stable homology of all 
the embeddings of braid groups into mapping class groups of 
orientable or non-orientable surfaces constructed earlier. 

\medskip

The Harer-Ivanov homology stability theorem states that the
embedding $\Sigma_{g,1} \to \Sigma _{g+1, 1}$ induces an isomorphism 

$$
H_* (\Gamma _{g,1}) \longrightarrow H_* (\Gamma _{g+1,1})
$$

\noindent
in a range of degrees, called the {\it stable range}. 
This range has recently been improved to $* \leq 2(g-1)/3$
by Boldsen [Bo] and Randal-Williams [RW].

\begin{theorem}
{\it In the stable range, the maps 
$R \circ \gamma$, $L \circ \varphi$, $\phi$, $\varphi^+$, $L' \circ \phi$, 
and $R' \circ \phi$ induce the zero map in any reduced, generalised homology theory.} 
\end{theorem}

\begin{proof} 
We sketch the argument here and refer for more details to [SoT], [SeT].
In section 4 we showed that the maps $R \circ \gamma$, $L \circ \varphi$, $\phi$ and $\varphi^+$ 
are maps of
$\cD$-algebras. After taking classifying spaces and 
group completion they induce therefore maps of double loop spaces  
\footnote{ The group completion of $\coprod _{m \geq 1} B(B_{2m}) $ consists 
of all the even components in $\Omega^2 S^2$, and the argument goes through.} 

$$
\Omega B(\coprod _{g>0} B( B_g)) 
\simeq \bZ \times B(B_\infty)^+ \simeq \Omega ^2 S^2
\longrightarrow
\Omega B(\coprod_{g\geq 0} B\Gamma _{g,1}) 
\simeq \bZ \times B\Gamma _\infty ^+.
$$

\noindent
Here $B_\infty = \lim _{g \to \infty} B_g$ and $\Gamma _\infty = 
\lim_{g \to \infty} \Gamma _{g,1}$
are the infinite braid and mapping class groups and
$X^+$ denotes the Quillenization of the space $X$. 
As $B(B_\infty)^+ \simeq \Omega ^2_0 S^2 \simeq \Omega ^2 S^3$ is the 
free object on the circle in the category of double loop spaces, 
on a connected component
these maps are determined by their restriction
to the circle. But these restrictions have to be homotopic to the constant map
as 
$B(\Gamma _\infty)^+$ is simply connected. Hence maps 
$B(B_\infty)^+ \to B(\Gamma _\infty)^+$ that are maps of double loop spaces
are null-homotopic. In particular, they induce the zero
map in any reduced, generalised  homology theory. 
Finally, the mapping class groups satisfy (ordinary) homology stability. 
By an application of the
Atiyah-Hirzebruch spectral sequence, the statement of the theorem
follows for the first four maps, including  $\phi$, and 
hence for $L' \circ \phi$ and $R' \circ \phi$.
\end{proof}

We now turn our attention to the mapping class group of non-orientable
surfaces and the embedding  $\varphi: B_g \to \cN_{g,1}$ defined in (2.3).
The commutator subgroup of $\cN_{g}$ is generated by
Dehn twists around two-sided curves. For $g \geq 7$ it
has index two and  
thus $H_1 (\cN_{g}) = \bF_2$ and is in particular not trivial; see [K]. 

\medskip

The mapping class groups $\cN_g$ also satisfy homology stability:

$$
H_* (\cN_g) = H_* (\cN_{g,1}) = H_* (\cN_{g+1, 1}).
$$

\noindent
Here  $* \leq  (g-3)/3$ for the first equality and $* \leq g/3$ for the second,
see [Wa2] and [RW]. 


\smallskip

If $\sigma$ is  a standard generator of the braid 
group, its image under $\varphi$ interchanges  
two cross caps in $N_{g,1}$. Therefore,
it is not in the index two subgroup of $\cN_{g,1}$ generated by Dehn twists
around two-sided curves.
Indeed, the 
product of $\varphi(\sigma )$ with  the Dehn twist around the two-sided
curve that goes once through each cross cap is a cross-cap slide, see [S].
Hence, the map induced by $\varphi$ on
the first homology groups is surjective and not trivial. More generally
we have the following result.

\begin{theorem}
{\it Let $g \geq 7$ and $0 < * \leq g/3$. 
When $\bF = \bQ$ or $\bF = \bF_p$ for an odd prime $p$,
the map
$$
\varphi_* : H_*(B_g ; \bF) \longrightarrow  H_*(\cN_{g,1} ; \bF)
$$
is zero, while   for $\bF = \bF_2$ it is an injection.}
\end{theorem}

\begin{proof}
The basic idea of the proof is similar to that used in Theorem 5.1 but 
we need to  use also some  quite technical results from [Wa], [T1] and [T2].
We sketch the argument.

It is well-known that the map from the braid to the symmetric group
induces after taking 
classifying spaces, stabilisation and Quillenization the canonical map

$$
\bZ \times BB_\infty^+ \simeq
\Omega^2 S^2 \longrightarrow \Omega^\infty S^\infty \simeq 
\bZ \times B\Sigma_\infty^+
$$ 

\noindent
from the free object generated by $S^0$ in the category of double loop spaces 
to the corresponding one in the category of infinite loop spaces.   
In homology with $\bF_2$ coefficients it 
induces an inclusion and it is zero in reduced homology with 
field coefficients of  characteristic
other than 2.

By the main theorem of [T1], the double loop space 
structure on $\bZ \times B\cN_\infty^+$ 
defined by the $\cD$-algebra structure on $\cN$
extends to an infinite loop space structure. 
This implies that the map $\Omega ^2 S^2 \to \bZ \times B \cN_\infty^+$
induced by $\varphi$ 
factors through $\Omega^\infty S^\infty$ via 
the above map.

Using cobordism categories of non-orientable surfaces one can show
that there is another infinite loop space structure on 
$\bZ \times B \cN^+_\infty$.  
By a theorem of Wahl these two infinite loop space 
structures are the same up to homotopy. 
To be more precise, Wahl shows in [Wa1] that the two constructions lead 
to the 
same 
infinite loop space structures up to homotopy in the orientable case, 
i.e., for 
$\bZ \times B \Gamma^+_\infty$. Her argument goes through verbatim to prove 
the same result 
for non-orientable surfaces.  
Thus, the map of infinite loop spaces
$\Omega^\infty S^\infty \to \bZ \times B \cN^+_\infty$
here is up to homotopy the same as the one used in [T2]. 
In [T2] we showed however that this map 
has a splitting up to homotopy and,   
in particular, induces an injection in homology. 

Combining all this we have proved that  the
composition 

$$
\bZ \times BB_\infty ^+ \simeq
\Omega^2 S^2 \longrightarrow \Omega ^\infty S^\infty \longrightarrow 
\bZ \times B \cN ^+_\infty
$$

\noindent
induces an injection on $\bF_2$-homology and is trivial in reduced homology
for $\bF= \bQ$
or $ \bF_p$, $p$ odd. As $ B_g \to B_\infty$ induces an injection
in homology with any field coefficients, see [CML], and 
by homology stability of the non-orientable mapping class group 
the theorem follows.
\end{proof}

\bigskip
\section{Calculations in  unstable homology}

In this section we examine homomorphism induced  
in unstable homology with field coefficients by our embeddings. 
We restrict our discussion to the orientable case 
though a  similar analysis goes through also in the non-orientable case. 

\bigskip

Consider any map

\begin{equation}
\alpha_* : H_* (B_m, \bF) \longrightarrow H_*(\Gamma _{g,b}, \bF)
\end{equation}

\noindent
induced by a homomorphism $\alpha : B_m \to \Gamma_{g,b}$
that is part of a $\cD$-algebra map; here $b = 1, 2$ and $\bF$ is any field.
The main fact we will be using is that the homology of the braid group is 
generated by classes of degree one when taking the $\cD$-algebra structure 
into account.

\bigskip
\noindent
{\bf 6.1. The rational case:}
Recall from [CLM] that
for $m > 1$ the $\bQ$-homology of $B_m$ is of rank one in degrees $0$ and $1$, 
and zero otherwise. 
Thus rationally, the braid groups have the homology of a circle.
We recall

\begin{equation}
H_1(\Gamma_{2,2}) = H_1(\Gamma_{2,1}) = \bZ / 10 \bZ, 
\quad \text{ and } \quad \; H_1(\Gamma_{g,1}) = 0 \, \text{ when } g \geq 3. 
\end{equation}

\noindent
%
The first identity follows from the stability results [Bo] and [RW]; 
and the two computations are well-known. 
Thus $\alpha_*$ is trivial in rational homology.

\bigskip

We now turn to fields of finite characteristic. 
Recall from [CLM] that for $m > 1$ 
the $\bF_p$-homology is generated by $H_1 (B_m; \bF_p) = \bF_p$ and
the homology operations induced from the action of $\cD$. 
These  operations are the product,
the first Dyer-Lashof operation $Q$, and in the case of odd primes, 
the combination with the Bockstein 
operator $\beta Q$. More precisely,
for $x \in H_*(B_m; \bF_p)$, the operation $Q$ is defined by the formula

\begin{eqnarray}
Q(x) &=& \theta_* (e_1 \otimes x \otimes x)    \quad \text{ for } \quad p=2 \\
Q(x) &=& \theta_* (e_{p-1} \otimes x^p)     \,\:\quad \text{ for } \quad p>2 
\end{eqnarray}

\noindent
where $e_1$ is of degree $1$, $e_{p-1}$ of degree $p-1$, 
and $\theta_*$ is induced by  the action $\theta = \theta_{\cB}$ of
$\cD$ on $\cB$ as defined in section 4.

\bigskip

\noindent
{\bf 6.2. The case $p$ even:} 
Quoting [CLM, p. 347], 
the $\bF_2$-homology of the braid group can be described as

$$
H_*(B_m; \bF_2) = \bF_2 [ x_i] / I
$$

\noindent
where  $x_i= Q(x_{i-1})$ is of degree $2^i-1$ and $I$ is the ideal generated by
all monomials $x_{i_1}^{k_1} \dots x_{i_t}^{k_t}$ 
such that $\sum_{j=1} ^t \, k_j \, 2^{i_j} > m$.  In particular, 

$$
x_i = 0 \quad \text { if } \quad 2^i > m.
$$

\bigskip

\noindent
{\bf 6.3. The case $p$ odd:}
Similarly, by [CLM, p. 347], for $p$ odd 
the $\bF_p$-homology of the braid group is a polynomial algebra
on generators $\lambda$, $y_i$, and $ \beta y_i$ 
of degrees $1$, $2 p^i-1$, and
$2p^i - 2$, modulo some  ideal $J$ which includes 
$y_i$ whenever $2p^i >m$. 
Furthermore, $y_1 = Q(\lambda)$ and $y_{i+1} = Q(y_i)$. 

\bigskip

We consider now the image of the generators $x_i$ and $y_i$ under 
the map $\alpha_*$.
As an  immediate consequence of Theorem 5.1, and  
using the best homology stability range available, we have

\begin{align}
\alpha_*(x_i) &= 0 \;\;\; \text{ for } \;\; (3 \cdot 2^i - 1)/2 \leq g   \;\;\;  \text{ and } \;\; p = 2 \\
\alpha_*(y_i) &= 0 \;\;\; \text{ for } \;\; (6 \, p^i - 1)/2 \leq g      \;\;\;  \text{ and } \;\; p > 2.
\end{align}

\bigskip

We now explain in two examples how the Dyer-Lashof algebra structure can be used 
to deduce similar results independent of Theorem 5.1, which in some cases
lead to stronger vanishing results.


\bigskip

\noindent
{\bf 6.4. Example:}
Consider the operadic embedding (2.5) when $m=g$, $b=1$ and
$\alpha$ is

$$  
\varphi^+ : B_g \longrightarrow \Gamma _{g,1}.
$$ 

\noindent
The $\cD$-algebra structure in this case is given by gluing the boundary 
of the surfaces $\Sigma_{g_i, 1}$ to the boundaries of the holes in the disk
$\Sigma_{0, k+1}$. The genus of the resulting surface is simply the 
sum of the genera $g_i$. 

\bigskip

Observe that to use the $\cD$-algebra structure, 
the genus of 
the surface
corresponding to the target group has to be large enough to be able 
to decompose it.
The equation $\varphi^+_* (Q(z)) = Q (\varphi^+_* (z))$ 
gives rise to the following  inductive formula. 
Assume 
$\varphi^+_* (x_{i+1}) = 0$  or $\varphi^+_* (y_i) = 0$ when the genus of the target 
surface is at least $d_i$. Then using 
(6.3) and (6.4) we can conclude that 
$$
d_i = p \, d_{i-1} \quad \text {and thus } \quad d_i = p^i \, d_0,
$$
where, by (6.2), $d_0 =3 $ when $p= 2,5$, and $d_0 =2$ when $p \neq 2,5$.  
Thus we have that

\begin{align}
\varphi^+_* (x_i) &= 0 \;\;\; \text{ if } \;\;  3 \cdot 2^{i-1} \leq g \;\;\; \text{ and } \;\; p = 2, \\
\varphi^+_* (y_i) &= 0 \;\;\; \text{ if } \;\;  3 \, p^i \leq g       \;\;\; \text{ and } \;\; p = 5, \\
\varphi^+_* (y_i) &= 0 \;\;\; \text{ if } \;\;  2 \, p^i \leq g       \;\;\; \text{ and } \;\; p \neq 5.  
\end{align}

\noindent
This gives an improvement on (6.5) and (6.6) for primes other than 2 and 5. 

\medskip

As $x_i= 0$ for $2^i>m$ and $y_i=0$ for $2 \, p^i>m$, 
our computations show that the map in homology induced by $\varphi^+$
is zero in unstable homology of positive degree for 
characteristics $p \neq 2, 5$. When $p=2$ or $p=5$, we cannot always
decide with the above methods 
whether the top dimensional $x_i$ and  $y_i$ classes are mapped to zero or not.

\bigskip

\noindent
{\bf 6.5. Example:}
We now consider the geometric embedding (1.1) when  $m = 2g + 2$, $b = 2$ and $\alpha$ is
$$
\phi : B_{2g+2} \longrightarrow \Gamma _{g,2}.
$$
The $\cD$-algebra structure in this case glues the surfaces $\Sigma_{g_i,2}$
to two disks $\Sigma _{0, k+1}$ with $k$ holes. The resulting surface
has genus the sum of the genera $g_i$ plus $k-1$. Thus,
the inductive formula for $d_i$ is given by

$$
d_{i} = p \, d_{i-1} + p - 1 = p^i \, (d_0 + 1) - 1.
$$

\noindent
By (6.2) 
we have $d_0 = 3$ and $d_i = 4 \, p^i - 1$ for $p = 2, 5$,
and we have $d_0 = 2$ and $d_i = 3 \, p^i - 1$ for $p \neq 2, 5$.

\noindent
The $d_i$ are growing faster here then in Example 6.4. 
Indeed, as is easily checked, any class $x_i$ and $y_i$
that in this way can be shown to vanish under $\phi_*$ 
is already in the stable range.
Thus no extra information can be gained in addition to what is 
known by Theorem 5.1; compare (6.5) and (6.6).

\medskip

In particular, we cannot determine with our methods here 
whether 
(a) the top three $x_i$ classes
and whether
(b) 
the top two $y_i$ classes for $p = 3$ resp. 
the top $y_i$ class for $p \neq 3$
vanish under $\phi_*$. 
We have
$x_i \neq 0$ when $2^{i-1} -1 \leq g$, 
while $\phi_*(x_i) = 0$ 
when $2^{i+2} -1 \leq g$ 
and similarly  
$y_i \neq 0 $ when $p^{i-1} - 1 \leq g$, 
while $\phi_* (y_i) = 0$ 
when $4 p^i -1 \leq g$ in case $p =5$,
or when $3 p^i - 1 \leq g$ in case $p \neq 5$. 

\medskip

\noindent
Therefore the conclusions drawn in [SeT; Corollary 4.1]
(and [C; Corollary 2.7]) are too strong:
{\em It   remains an  open question whether for $g \geq 3$ the homomorphism
$\phi : B_{2g+2} \to \Gamma_{g,2}$ induces the zero map in reduced $\bF_p$-homology  
for all primes $p$.}

\bigskip

In our explicit calculations above we have concentrated on the generators $x_i$
 and $y_i$. But a similar analysis can  be given to determine  when
the image of  products is trivial. 

\bigskip

Finally, an analogous study can be given for $\cD$-algebra maps 
from braid groups to mapping class groups of non-orientable mapping class 
groups in the case when $p\neq 2$. 
Note, in that case $H_1 (\cN _g; \bF_p) = 0$ for $g\geq 7$ by [K].

 

\bigskip
\bigskip
\bigskip

\noindent
\small Carl-Friedrich Bödigheimer \\  
\small Mathematisches Institut\\
\small Universit\"at Bonn \\
\small Endenicher Allee 60  \\
\small 53115 Bonn, Germany  \\   
\small boedigheimer@math.uni-bonn.de  \\[0.5cm]
\small Ulrike Tillmann \\
\small Mathematical Institute \\
\small Oxford University \\
\small 24-29 St Giles \\
\small Oxford OX1 3LB, United Kingdom \\
\small tillmann@maths.ox.ac.uk

\end{document}

%% file: fig1.pstex_t
\begin{picture}(0,0)%
\includegraphics{fig1.pstex}%
\end{picture}%
\setlength{\unitlength}{2072sp}%
\begingroup\makeatletter\ifx\SetFigFont\undefined%
\gdef\SetFigFont#1#2#3#4#5{%
  \reset@font\fontsize{#1}{#2pt}%
  \fontfamily{#3}\fontseries{#4}\fontshape{#5}%
  \selectfont}%
\fi\endgroup%
\begin{picture}(9273,3166)(661,-3490)
\put(676,-2131){\makebox(0,0)[lb]{\smash{{\SetFigFont{6}{7.2}{\rmdefault}{\mddefault}{\updefault}{\color[rgb]{0,0,0} }%
}}}}
\put(2836,-2401){\makebox(0,0)[lb]{\smash{{\SetFigFont{6}{7.2}{\rmdefault}{\mddefault}{\updefault}{\color[rgb]{0,0,0}$a_3$}%
}}}}
\put(1981,-2626){\makebox(0,0)[lb]{\smash{{\SetFigFont{6}{7.2}{\rmdefault}{\mddefault}{\updefault}{\color[rgb]{0,0,0}$a_2$}%
}}}}
\put(1126,-2401){\makebox(0,0)[lb]{\smash{{\SetFigFont{6}{7.2}{\rmdefault}{\mddefault}{\updefault}{\color[rgb]{0,0,0}$a_1$}%
}}}}
\put(8146,-2626){\makebox(0,0)[lb]{\smash{{\SetFigFont{6}{7.2}{\rmdefault}{\mddefault}{\updefault}{\color[rgb]{0,0,0}$a_{2g}$  }%
}}}}
\put(3736,-2626){\makebox(0,0)[lb]{\smash{{\SetFigFont{6}{7.2}{\rmdefault}{\mddefault}{\updefault}{\color[rgb]{0,0,0}$a_4$}%
}}}}
\put(9046,-2446){\makebox(0,0)[lb]{\smash{{\SetFigFont{6}{7.2}{\rmdefault}{\mddefault}{\updefault}{\color[rgb]{0,0,0}$a_{2g+1}$}%
}}}}
\end{picture}%

%% file: fig2.pstex_t
\begin{picture}(0,0)%
\includegraphics{fig2.pstex}%
\end{picture}%
\setlength{\unitlength}{2072sp}%
\begingroup\makeatletter\ifx\SetFigFont\undefined%
\gdef\SetFigFont#1#2#3#4#5{%
  \reset@font\fontsize{#1}{#2pt}%
  \fontfamily{#3}\fontseries{#4}\fontshape{#5}%
  \selectfont}%
\fi\endgroup%
\begin{picture}(9758,2971)(882,-3854)
\put(5401,-1951){\makebox(0,0)[lb]{\smash{{\SetFigFont{6}{7.2}{\rmdefault}{\mddefault}{\updefault}{\color[rgb]{0,0,0}{\Large $\sigma$}}%
}}}}
\end{picture}%

%% file: fig3.pstex_t
\begin{picture}(0,0)%
\includegraphics{fig3.pstex}%
\end{picture}%
\setlength{\unitlength}{4144sp}%
\begingroup\makeatletter\ifx\SetFigFont\undefined%
\gdef\SetFigFont#1#2#3#4#5{%
  \reset@font\fontsize{#1}{#2pt}%
  \fontfamily{#3}\fontseries{#4}\fontshape{#5}%
  \selectfont}%
\fi\endgroup%
\begin{picture}(7396,6684)(1298,-7183)
\put(2701,-4966){\makebox(0,0)[lb]{\smash{{\SetFigFont{12}{14.4}{\rmdefault}{\mddefault}{\updefault}{\color[rgb]{0,0,0}$a_g$}%
}}}}
\put(3151,-3931){\makebox(0,0)[lb]{\smash{{\SetFigFont{12}{14.4}{\rmdefault}{\mddefault}{\updefault}{\color[rgb]{0,0,0}$a_3$}%
}}}}
\put(3466,-3301){\makebox(0,0)[lb]{\smash{{\SetFigFont{12}{14.4}{\rmdefault}{\mddefault}{\updefault}{\color[rgb]{0,0,0}$a_2$}%
}}}}
\put(2251,-3391){\makebox(0,0)[lb]{\smash{{\SetFigFont{12}{14.4}{\rmdefault}{\mddefault}{\updefault}{\color[rgb]{0,0,0}$a_1$}%
}}}}
\end{picture}%

%% file: fig4.pstex_t
\begin{picture}(0,0)%
\includegraphics{fig4.pstex}%
\end{picture}%
\setlength{\unitlength}{2072sp}%
\begingroup\makeatletter\ifx\SetFigFont\undefined%
\gdef\SetFigFont#1#2#3#4#5{%
  \reset@font\fontsize{#1}{#2pt}%
  \fontfamily{#3}\fontseries{#4}\fontshape{#5}%
  \selectfont}%
\fi\endgroup%
\begin{picture}(9603,1866)(442,-2794)
\put(2071,-2491){\makebox(0,0)[lb]{\smash{{\SetFigFont{6}{7.2}{\rmdefault}{\mddefault}{\updefault}{\color[rgb]{0,0,0}$a$}%
}}}}
\put(8011,-1501){\makebox(0,0)[lb]{\smash{{\SetFigFont{6}{7.2}{\rmdefault}{\mddefault}{\updefault}{\color[rgb]{0,0,0}$J(a)$}%
}}}}
\end{picture}%